\documentclass[a4paper,12pt]{article}
\usepackage{comment}
\usepackage{cite}
\usepackage{amsmath}
\usepackage{amssymb}
\usepackage{amsfonts}
\usepackage[T1]{fontenc}
\usepackage[utf8]{inputenc}
\usepackage{graphicx}
\usepackage{fancyhdr}
\usepackage{float}
\usepackage{xcolor}
\usepackage{authblk}
\usepackage{mathrsfs}
\usepackage{empheq}
\usepackage[hyphens]{url}
\usepackage{hyperref}
\usepackage[]{breakurl}
\usepackage[]{amsthm}
\usepackage{tikz}
\usetikzlibrary{decorations.markings}

\usepackage{geometry}
\newtheorem{theorem}{Theorem}

\begin{document}

\title{Generalized relations between arithmetic functions}

\author[$\dagger$]{Jean-Christophe {\sc Pain}\footnote{jean-christophe.pain@cea.fr}\\
\small
$^1$CEA, DAM, DIF, F-91297 Arpajon, France\\
$^2$Universit\'e Paris-Saclay, CEA, Laboratoire Mati\`ere en Conditions Extr\^emes,\\
F-91680 Bruy\`eres-le-Ch\^atel, France
}

\date{}

\maketitle

\begin{abstract}
The aim of this article is to present in a self–contained way identities arising in elementary number theory, among which the following one:
\[
\sum_{d\mid n}\frac{\mu^2(d)}{\varphi(d)\,d^s}=\prod_{p\mid n}\left(1+\frac{1}{(p-1)p^s}\right).
\]
This formula expresses a non–trivial divisor sum involving the M\"obius function $\mu$ and Euler's totient function $\varphi$ as a simple and explicit multiplicative expression. This is a generalization of the remarkable Dineva formula, which corresponds to $s=0$ and gives $n/\varphi(n)$ on the right-hand side. We explain why only squarefree divisors are involved, show how multiplicativity naturally comes into play, and interpret the identity as a finite Euler product. Beyond this one–parameter family of generalizations, we describe a general method for constructing similar formulas and present several examples. Finally, we reformulate these identities in terms of partial zeta functions, thus emphasizing their close relationship with the classical theory of Euler products and the Riemann zeta function. The connection with the Selberg sieve is briefly outlined.
\end{abstract}

\section{Introduction}\label{sec1}

Divisor sums and multiplicative functions form one of the fundamental tools of elementary and analytic number theory. Many classical identities reveal that arithmetic functions admit two complementary descriptions: on the one hand, as sums over divisors, and on the other hand, as products indexed by prime numbers. These two points of view are linked by the principle of multiplicativity and are at the heart of Euler's discovery that prime numbers control the structure of the integers. A very well–known example is the identity
\[
\sum_{d\mid n}\varphi(d)=n,
\]
which reflects the fact that the integers between $1$ and $n$ can be grouped according to the value of their greatest common divisor with $n$ \cite{Apostol}. Another, less classical but equally elegant, formula is
\[
\sum_{d\mid n}\frac{\mu^2(d)}{\varphi(d)}=\frac{n}{\varphi(n)},
\]
which is often referred to as the Dineva formula \cite{Dineva}. It shows that a weighted sum over the squarefree divisors of $n$ collapses to a very simple expression involving only $n$ and its Euler totient.

The purpose of this article is to propose generalizations of such a formula, as well as a general technique to derive new relations. In section \ref{sec2}, we explain why this identity is natural and how it follows almost automatically from multiplicativity. In section \ref{sec3}, we show how it fits into a more general framework, and propose new identities. We also see, in the same section, that it admits a formulation in terms of finite Euler products and partial zeta functions, which makes its structure particularly transparent.

\section{The Dineva formula}\label{sec2}

The Dineva identity reads
\[
\sum_{d\mid n}\frac{\mu^2(d)}{\varphi(d)}=\frac{n}{\varphi(n)},
\]
and the M\"obius function $\mu(d)$ is defined by
\[
\mu(d)=
\begin{cases}
1 & \text{if } d=1,\\
(-1)^k & \text{if } d \text{ is a product of } k \text{ distinct primes},\\
0 & \text{if } d \text{ is divisible by the square of a prime}.
\end{cases}
\]
Therefore, $\mu^2(d)$ takes only the values $0$ and $1$, and one has
\[
\mu^2(d)=
\begin{cases}
1 & \text{if } d \text{ is squarefree},\\
0 & \text{otherwise}.
\end{cases}
\]
Finally, $\varphi(d)$ denotes Euler's totient function, that is, the number of integers between $1$ and $d$ which are relatively prime to $d$. With this interpretation, the sum
\[
\sum_{d\mid n}\frac{\mu^2(d)}{\varphi(d)}
\]
actually runs only over the squarefree divisors of $n$. The Dineva formula thus states that if one sums the reciprocals of $\varphi(d)$ over all squarefree divisors $d$ of $n$, the result is simply $n/\varphi(n)$. The proof relies on a fundamental idea: multiplicativity. Recall that an arithmetic function $f$ is said to be multiplicative if
\[
f(ab)=f(a)f(b)
\quad\text{whenever}\quad \gcd(a,b)=1.
\]
Both $\mu^2(d)$ and $\varphi(d)$ are multiplicative functions, and it follows that the left–hand side of the Dineva formula defines a multiplicative function of $n$. The right–hand side, $n/\varphi(n)$, is also multiplicative. Hence, it suffices to verify the identity when $n$ is a power of a prime. Let $n=p^k$, where $p$ is prime and $k\ge 1$. The divisors of $p^k$ are
\[
1,p,p^2,\dots,p^k,
\]
but among them only $1$ and $p$ are squarefree. Therefore,
\[
\sum_{d\mid p^k}\frac{\mu^2(d)}{\varphi(d)}
=\frac{\mu^2(1)}{\varphi(1)}+\frac{\mu^2(p)}{\varphi(p)}
=1+\frac{1}{p-1}
=\frac{p}{p-1}.
\]
On the other hand, one has
\[
\frac{n}{\varphi(n)}
=\frac{p^k}{\varphi(p^k)}
=\frac{p^k}{p^k-p^{k-1}}
=\frac{p}{p-1},
\]
and thus the two sides agree for all prime powers. By multiplicativity, they agree for every positive integer $n$, which completes the proof. This identity can be rewritten in Euler product form. Indeed, since only squarefree divisors contribute, we may write
\[
\sum_{d\mid n}\frac{\mu^2(d)}{\varphi(d)}
=\prod_{p\mid n}\left(1+\frac{1}{\varphi(p)}\right)
=\prod_{p\mid n}\left(1+\frac{1}{p-1}\right)
=\prod_{p\mid n}\frac{p}{p-1}
=\frac{n}{\varphi(n)}.
\]
This expression shows clearly how each prime divisor of $n$ contributes independently to the value of the sum.

\section{A general construction principle and new identities}\label{sec3}

\subsection{Main steps of the derivation}\label{subsec31}

The Dineva formula is not an isolated phenomenon. It actually fits into a general mechanism for constructing identities of the form
\[
F(n)=\sum_{d\mid n}\mu^2(d)\,H(d),
\]
where $F$ is a multiplicative function. Since the term $\mu^2(d)$ vanishes whenever $d$ is not squarefree, the sum only depends on the distinct prime factors of $n$. In other words, $F(n)$ is determined by the radical of $n$, namely $\prod_{p\mid n} p$.

Suppose that $F$ is multiplicative and can be written as
\[
F(n)=\prod_{p\mid n} f(p),
\]
for some function $f$ defined on prime numbers. If we write
\[
f(p)=1+g(p),
\]
then expanding the product shows that
\[
F(n)
=\prod_{p\mid n}\bigl(1+g(p)\bigr)
=\sum_{d\mid n}\mu^2(d)\prod_{p\mid d} g(p).
\]
Indeed, the squarefree divisors $d$ of $n$ correspond exactly to the choices of subsets of the set of primes dividing $n$. For the Dineva formula, we take
\[
g(p)=\frac{1}{p-1},
\]
which yields
\[
\prod_{p\mid d} g(p)=\frac{1}{\varphi(d)}
\quad\text{for every squarefree } d.
\]

\subsection{Examples}\label{subsec32}

If we choose the function
\[
f(p)=1+\frac{1}{p},
\]
then we have
\[
\prod_{p\mid n}\left(1+\frac{1}{p}\right)
=\sum_{d\mid n}\frac{\mu^2(d)}{d}.
\]
This identity shows that the sum of $1/d$ over all squarefree divisors $d$ of $n$ admits a simple multiplicative description. More generally, for any real parameter $s$, one has
\[
\sum_{d\mid n}\frac{\mu^2(d)}{d^s}
=\prod_{p\mid n}(1+p^{-s}),
\]
and it turns out that this identity is closely related to the Riemann zeta function, since
\[
\prod_{p}(1+p^{-s})=\frac{\zeta(s)}{\zeta(2s)}.
\]
When infinite Euler products are involved, we implicitly assume $s>1$. Thus, the finite product over primes dividing $n$ can be viewed as a truncated version of a classical Euler product. Indeed, let us introduce the \emph{partial zeta function} associated with $n$:
\[
\zeta_n(s)=\prod_{p\mid n}\frac{1}{1-p^{-s}}.
\]
This is simply the Euler product for $\zeta(s)$ restricted to the primes dividing $n$. The above formula can therefore be recast into
\[
\sum_{d\mid n}\frac{\mu^2(d)}{d^s}=\frac{\zeta_n(s)}{\zeta_n(2s)}.
\]

\subsection{A generalization of the Dineva formula} \label{subsec33}

\begin{theorem}

For any real number  $s$, we have
\[
\sum_{d\mid n}\frac{\mu^2(d)}{\varphi(d)\,d^s}
=\prod_{p\mid n}\left(1+\frac{1}{(p-1)p^s}\right).
\]

\end{theorem}

\begin{proof}

Let us define
\[
F(n)=\sum_{d\mid n}\frac{\mu^2(d)}{\varphi(d)\,d^s}.
\]
We first show that $F$ is multiplicative. Indeed, the function $\mu^2(d)$ is multiplicative, $\varphi(d)$ is multiplicative, and $d^s$ is completely multiplicative. Hence their quotient
\[
\frac{\mu^2(d)}{\varphi(d)\,d^s}
\]
is multiplicative, and therefore the divisor sum defining $F(n)$ is a multiplicative function of $n$.

Consequently, it suffices to evaluate $F(n)$ on prime powers. Let $n=p^k$ with $p$ prime and $k\ge1$. The divisors of $p^k$ are
\[
1,p,p^2,\dots,p^k,
\]
but among them only $1$ and $p$ are squarefree. Hence
\[
F(p^k)
=\frac{\mu^2(1)}{\varphi(1)\,1^s}
+\frac{\mu^2(p)}{\varphi(p)\,p^s}.
\]
Since
\[
\mu^2(1)=1,\qquad \varphi(1)=1,
\]
and
\[
\mu^2(p)=1,\qquad \varphi(p)=p-1,
\]
we obtain
\[
F(p^k)=1+\frac{1}{(p-1)p^s}.
\]

Thus, for every prime $p$ and every integer $k\ge1$, one can write
\[
F(p^k)=1+\frac{1}{(p-1)p^s}.
\]
Now let
\[
n=\prod_{p\mid n} p^{\alpha_p}
\]
be the prime factorization of $n$. By multiplicativity of $F$, we conclude that
\[
F(n)=\prod_{p\mid n} F(p^{\alpha_p})
=\prod_{p\mid n}\left(1+\frac{1}{(p-1)p^s}\right),
\]
and therefore
\[
\sum_{d\mid n}\frac{\mu^2(d)}{\varphi(d)\,d^s}
=\prod_{p\mid n}\left(1+\frac{1}{(p-1)p^s}\right),
\]
which is exactly the desired identity.
\end{proof}

Starting with the left-hand side, one gets
\[
1 + \frac{1}{(p-1)p^s} = \frac{(p-1)p^s + 1}{(p-1)p^s}.
\]
Expanding and simplifying gives
\[
\frac{p^{s+1} - p^s + 1}{(p-1)p^s} = \frac{1 - p^{-1} + p^{-(s+1)}}{1 - p^{-1}}
\]
which yields the alternate expression
\[
\sum_{d\mid n}\frac{\mu^2(d)}{\varphi(d)\,d^s}=\prod_{p\mid n}\frac{1 - p^{-1} + p^{-(s+1)}}{1 - p^{-1}}.
\]
Since $\zeta_p(1)=\frac{1}{1-p^{-1}}$, this factor can also be written as
\[
\sum_{d|n} \frac{\mu^2(d)}{\varphi(d)d^s} = \prod_{p|n} \left( 1 + \frac{\zeta_p(1)}{p^{s+1}} \right),
\]
with
\[
\zeta_p(s)=\frac{1}{1-p^{-s}}.
\]
Other Dirichlet series are recalled and discussed in Appendix A.

\section{Conclusion}

The Dineva formula provides a striking example of how simple and elegant identities emerge from the multiplicative structure of arithmetic functions. Although it involves several classical objects of number theory, like the M\"obius function or Euler's totient function, divisor sums, and Euler products, it ultimately reduces to a transparent and elementary statement. More importantly, the formula is not an isolated curiosity. We have shown that it belongs to a whole family of identities obtained by decomposing multiplicative functions into elementary contributions at each prime. The general construction principle explained here shows how such formulas can be systematically produced. Finally, the introduction of partial zeta functions highlights the close connection between these finite identities and the analytic theory of the Riemann zeta function. In this sense, the new formulas can be seen as finite and purely arithmetic reflections of much deeper analytic structures. There are strong connections between the simple formulas discussed here and the Selberg sieve \cite{Selberg1947,Hooley1976} (see Appendix B). 

In the future, we plan to investigate similar sum rules involving other arithmetic functions such as the number of divisors \cite{Pain2025}, the sum of divisors \cite{Sandor2015} or gcd-sum functions \cite{Toth2010} as well as generalizations of the M\"obius function itself \cite{Sandor2003}.

\section*{Appendix A: On some related Dirichlet series}

For a multiplicative function $f$, the sum over the divisors of $n$ can be expressed as a product over the prime powers $p^k$ that exactly divide $n$ (denoted by $p^k \parallel n$):
\[
\sum_{d|n} \frac{f(d)}{d^s} = \prod_{p^k \parallel n} \left( \sum_{j=0}^k \frac{f(p^j)}{p^{js}} \right),
\]
where $p^k \parallel n$ means that $p^k$ divides $n$ but $p^{k+1}$ does not divide $n$, i.e. $p^k \mid n$ and $p^{k+1} \nmid n$.

Let us set $h(d) = 1/(\varphi(d)\,d^s)$. For a prime $p$, the local factor is:
\[
h(p)=\frac{1}{(p-1)p^s}.
\]
If $n$ is squarefree, since $h$ is multiplicative, we obtain:
\[
\sum_{d|n}\mu^2(d)\,h(d) = \prod_{p|n} \left( 1 + h(p) \right).
\]
Substituting the expression for $h(p)$ leads to the identity:
\[
\sum_{d\mid n}\frac{\mu^2(d)}{\varphi(d)\,d^s}=\prod_{p\mid n}\frac{1 - p^{-1} + p^{-(s+1)}}{1 - p^{-1}}.
\]

The M\"obius function $\mu(d)$ satisfies a well-known identity related to the inverse of the partial zeta function $\zeta_n(s)$ (introduced in section \ref{sec3}):
\[
\sum_{d|n} \frac{\mu(d)}{d^s} = \prod_{p|n} (1 - p^{-s}) = \frac{1}{\zeta_n(s)}.
\]

If $\sigma(m)$ represents the sum of divisors of $m$, then for a prime power $p^k$, one has:
\[
\sigma(p^k) = \sum_{j=0}^{k} p^j = \frac{p^{k+1}-1}{p-1}.
\]
The corresponding local generating function factor is given by:
\[
\sum_{k=0}^{\infty} \frac{\sigma(p^k)}{p^{ks}} = \frac{1}{(1-p^{-s})(1-p^{1-s})}.
\]
Taking the product over all $p|n$, we establish the following relationship with the partial zeta functions:
\[
\prod_{p \mid n} \left( \sum_{k=0}^{\infty} \frac{\sigma(p^k)}{p^{ks}} \right) = \zeta_n(s) \zeta_n(s-1).
\]
Furthermore, using the fact that $\zeta_n(1) = n/\varphi(n)$, the product of the partial zeta functions at $1$ and $s$ can be expressed as:
\[
\zeta_n(1)\zeta_n(s) = \frac{n\,\zeta_n(s)}{\varphi(n)} = \prod_{p \mid n} \frac{1}{(1-p^{-1})(1-p^{-s})}.
\]

\section*{Appendix B: Connection with the Selberg sieve}

The Selberg sieve is a method for estimating the size of a set of integers that are not divisible by any prime numbers belonging to a given set. The efficiency of this sieve relies on the maximization of a quadratic form, which naturally leads to the sum investigated in this article.

\subsection*{The minimization problem}

In the Selberg sieve, one seeks to minimize the quantity 
\[
Q = \sum_{n \le X} \left( \sum_{d \mid n, d \le R} \lambda_d \right)^2
\]
subject to the constraint $\lambda_1 = 1$, where $R$ is a truncation parameter. Using the theory of multiplicative functions, it can be shown that the optimal choice of the coefficients $\lambda_d$ involves the function
\[
g(d) = \frac{1}{\varphi(d)\,d^s} \quad \text{or, in the classical case, } \quad g(d) = \frac{1}{\varphi(d)}.
\]

\subsection*{Emergence of the Dineva sum}

The core of Selberg's method consists in introducing the function $V(z)$ defined by the sum over squarefree divisors, or its discrete version  
\[
J_n = \sum_{d \mid n} \mu^2(d)\,g(d).
\]
The quantity $V(z)$ represents the continuous analogue of the finite sum $J_n$. While $J_n$ is a divisor sum associated with a fixed integer $n$, $V(z)$ describes the global density of integers surviving the sieve up to the level $z$ \cite{Halberstam}. In the case of the sieve for prime numbers, where $g(p) = 1/(p-1)$ (corresponding to $f(p)=1$), the sum becomes:
\[
J_n = \sum_{d \mid n} \frac{\mu^2(d)}{\varphi(d)}.
\]
The generalization presented here (Theorem 1) with the parameter $s$ corresponds to a modification of the sieve density, where each weight is weighted by $d^{-s}$. The identity demonstrated:
\[
\sum_{d \mid n} \frac{\mu^2(d)}{\varphi(d) d^s} = \prod_{p \mid n} \left( 1 + \frac{1}{(p-1)p^s} \right)
\]
shows that this critical quantity remains perfectly multiplicative, which allows for the analytical calculation of the sieve bound.

\subsection*{Optimal weights}

The optimal Selberg coefficients $\lambda_d$ are then given by the formula:
\[
\lambda_d = \mu(d) \frac{J_{n/d}}{J_n}.
\]

Thanks to the Euler product structure established in this work, these coefficients simplify dramatically. If we denote
\[
h_s(p) = \frac{1}{(p-1)p^s} = g(p),
\]
then for squarefree $d$:
\[
\lambda_d = \mu(d) \prod_{p \mid d} \frac{1}{1 + h_s(p)}.
\]
This form demonstrates that the Selberg sieve weights decrease more rapidly as the parameter $s$ increases, directly influencing the convergence of the method.
\end{document}